\documentclass[11pt]{preprint}

\usepackage{amsmath,amsfonts,latexsym,amssymb,amscd,
  graphicx,enumerate,amsthm}
\usepackage[normalem]{ulem}
\usepackage{soul}

\usepackage{cancel}
\usepackage{hyperref}
\usepackage[usenames]{color}

\newcommand{\N}{{\mathbb N}}

\newcommand{\Oc}{\mathcal{O}}
\newcommand{\Pp}{\mathrm{P}}

\newcommand{\Dmf}{{\mathfrak{D}}}
\newcommand{\Umf}{{\mathfrak{U}}}
\newcommand{\Vmf}{{\mathfrak{B}}}

\newcommand{\E}{\mathsf{E}}

\newcommand{\R}{{\mathbb R}}

\newcommand{\Cc}{{\mathcal C}}

\newcommand{\eps}{{\varepsilon}}
\newcommand{\e}{{\varepsilon}}

\newcommand{\sgn}{\mathop{\mathrm{sign}}}
\newcommand{\dist}{\mathop{\mathrm{dist}}}

\newcommand{\tBe}{{\tilde\tau_\eps}}

\newtheorem{theorem}{Theorem}[section]
\newtheorem{proposition}{Proposition}[section]
\newtheorem{remark}{Remark}[section]
\newtheorem{lemma}{Lemma}[section]
\newtheorem{corollary}{Corollary}[section]

\numberwithin{equation}{section}

\renewenvironment{proof}[1][Proof]{% beginning of proof
{\noindent {\sc #1: }}
}{% end of proof
{{\hfill $\Box$ \smallskip}}
}
\makeatletter
\let\orgdescriptionlabel\descriptionlabel
\renewcommand*{\descriptionlabel}[1]{%
  \let\orglabel\label
  \let\label\@gobble
  \phantomsection
  \edef\@currentlabel{#1}%
  \let\label\orglabel
  \orgdescriptionlabel{#1}%
}
\makeatother

\title{Scaling limit for escapes from unstable equilibria in the vanishing noise limit: nontrivial Jordan block case}
\author{Yuri Bakhtin, Zsolt Pajor-Gyulai}
\institute{Courant Institute of Mathematical Sciences\\
New York University\\
New York, NY, USA\\
}

\begin{document}

\maketitle

\begin{abstract}
We consider white noise perturbations of a nonlinear dynamical system in the neighborhood of an unstable critical point with linearization given by a Jordan block of full dimension. 
For the associated exit problem, we study the joint limiting behavior of the exit location and exit time, in the vanishing noise limit. 
The exit typically happens near one of two special deterministic points associated with
the eigendirection, and we obtain several more terms in the expansion for the exit point. The leading correction term 
is deterministic
and logarithmic in the noise magnitude, while the random remainder satisfies a scaling limit. 
\end{abstract}
\section{Introduction}

Exit problems for random perturbations of dynamical systems form an important classical field
in the theory of stochastic processes. These problems provide a multitude of interesting
questions at the intersection of dynamical systems and stochastic analysis, and are tightly related to asymptotic analysis of 
linear second order parabolic and elliptic equations with a small parameter.

The most celebrated
results for exit problems are large deviation estimates for the exit location and exit time by Freidlin and Wentzell~(see, e.g., \cite{FW2012}), in the case of a domain
containing one or several stable equilibria where the dynamics exhibits metastable behavior.

There are situations though where the analysis at the level of large deviations is not sufficient, and one is forced to study distributional scaling limits for the exit distributions. 
In~\cite{Bak2010} and~\cite{Bak2011}, this kind of analysis was carried out
for diffusions near noisy heteroclinic networks, where multiple hyperbolic critical points (or, {\it saddle points}) of the deterministic dynamics are connected to each other by {\it heteroclinic orbits} (or, {\it connections}). If  small noise is present, a typical trajectory near such a network  spends a long time diffusing near the critical points, where the vector field is weak, eventually deciding between outgoing heteroclinic 
connections and following one of them until it reaches the neighborhood of the next saddle point.
  Consequently, a natural approach based on the strong Markov property was an iterative study of the exit from the neighborhood of the
saddles and the motion along heteroclinic paths. Early results in this direction~\cite{Kifer1981}, \cite{Eizenberg:MR749377}, \cite{Bak2008}, \cite{Day95}, \cite{Mikami1995} established that with high probability, the exit from a neighborhood of an unstable equilibrium happens along the manifold associated with the top eigenvalue $\lambda>0$ of the linearization of the system and that the leading order asymptotics of the exit time is deterministic and is of the order of $\lambda^{-1}\log\e^{-1}$. However, these results were not detailed enough to allow for an efficient iteration scheme.
The necessary refinement of the analysis of the exit distribution was developed in~\cite{Bak2010}, \cite{Bak2011} (see also \cite{AB2011} where a technical no-resonance requirement was lifted for planar systems). This led to the first rigorous mathematical description of non-Markovian limiting effects and other behaviors in such systems despite the existing nonrigorous studies in~\cite{ASK03}, \cite{AS1999}, \cite{SH1990}. For a recent survey on heteroclinic networks, see \cite{Field2015}.

In \cite{Bak2010}, \cite{Bak2011},  and \cite{AB2011}, it was assumed that the top eigenvalues of the linearizations of the system near the critical points were simple. It was obtained then that if one starts near the critical point (or its stable manifold), then in the vanishing noise limit, the exit distribution satisfies a scaling limit theorem with explicitly computed scaling exponent
and limiting distribution.

In this paper, we are interested in a situation where the geometric multiplicity of
the leading eigenvalue $\lambda>0$ is equal to~$1$ and the algebraic multiplicity equals the dimension of the unstable manifold.
 For simplicity, we exclude the presence of the stable manifold, although our analysis carries over to the hyperbolic situation with obvious modifications.   Namely, we consider a vector field in arbitrary dimension, with one fully unstable critical point and linearization given by a matrix whose Jordan form contains exactly one Jordan block of full dimension.
With random initial conditions close to the critical point, we consider 
the small white noise perturbation of this vector field and study the  limiting behavior of the
joint distribution of the exit point and exit time in the limit of vanishing noise. 

Curiously, the limiting behavior is more involved compared to the case of the leading eigenvalue
of algebraic
multiplicity one, where the exit point satisfies a simple limit theorem. Namely, in our setting, we obtain that for small values of the noise magnitude $\eps$, the exit happens near one of two points $q_+$ and $q_-$ associated
with the main direction of the Jordan basis and, near each of $q_{\pm}$, the random exit point $z_\eps$ can
be represented by the following expansion:
\begin{equation}
\label{eq:main_expansion}
z_\eps=q_{\pm}
+
\left(\frac{1}{\log \eps^{-1}}
+\frac{(d-1) \log\log\e^{-1}}{\log^2\e^{-1}}
+\frac{\eta}{\log^2\e^{-1}}\right)h^{\pm}_1\\
+\frac{1}{\log^2\eps^{-1}}h^{\pm}_2+o_{\Pp}\left(\frac{1}{\log^2\eps^{-1}}\right),
\end{equation}
for some deterministic vectors $h_1^{\pm},h_2^{\pm}$
and a random variable $\eta$. In other words, given the direction of the exit (``$+$'' or ``$-$''), the leading correction to $q_{\pm}$ is deterministic and equals 
\[\left(\frac{1}{\log \eps^{-1}}+\frac{(d-1) \log\log\e^{-1}}{\log^2\e^{-1}}\right)h^{\pm}_1,\]
while the remainder  
\[
\frac{1}{\log^2\eps^{-1}}\left(\eta h^{\pm}_1+h^{\pm}_2\right)+o_{\Pp}\left(\frac{1}{\log^2\eps^{-1}}\right)
\]
satisfies a scaling limit theorem. 

Moreover, we show that given the direction of the exit (``$+$'' or ``$-$''), the exit time satisfies
\begin{equation}\label{eq:tau-D-e-asymp-in-intro}
\tau_{\Dmf}^\e=\frac{1}{\lambda}\log\e^{-1}-\frac{d-1}{\lambda}\log\log\e^{-1}+\rho+C^{\pm}+o_{\Pp}(1),
\end{equation}
for a centered random variable $\rho$ that does not depend on the direction of the exit and deterministic constants $C^{\pm}$. In fact, this is also in
contrast with the case of the leading eigenvalue of algebraic  multiplicity $1$,
where the leading deterministic term is simply $\frac{1}{\lambda}\log\e^{-1}$. We note that 
 \eqref{eq:tau-D-e-asymp-in-intro} was first obtained in~\cite{Buterakos}
for the case where the drift contains no other terms except for the linear one given by the Jordan block.
The precise statements of our results are given in Section~\ref{sec:setting}.

We remark that, according to \eqref{eq:main_expansion}, the leading contributions to the deviation from $q_{\pm}$ happen along $h_1^{\pm},h_2^{\pm}$. In fact, our proof also shows how to compute smaller contributions along other directions. We also note that it is easy to obtain a generalization of our result for the case where the linearization has other eigenvalues besides the leading $\lambda$.

The paper is organized as follows. In Section~\ref{sec:setting}, we describe the setting and
the main result. The proof of the main result in  Section~\ref{sec:proof-of-main} is based on the analysis 
of the linearized system in Section~\ref{sec:proof-of-linear}.

{\bf Acknowledgment.} We would like to thank the referee for valuable constructive remarks. They  helped to improve the paper in various ways.
Yuri Bakhtin gratefully acknowledges partial support from NSF via grant
DMS-1460595.  

\section{Setting and main result}\label{sec:setting}
We will consider the family of stochastic differential equations
\begin{equation}\label{eq:SDE}
dX_{\eps}(t)=b\left(X_{\eps}(t)\right)dt+\eps\sigma\left(X_{\eps}(t)\right)dW(t),
\end{equation}
on a bounded domain $\Dmf_0\subseteq \mathbb{R}^d$, $d\in\N$. Our results are most meaningful
for $d\ge 2$, but we include $d=1$ for completeness.
The drift is given by
a  vector field~$b\in\mathcal{C}^{2}(\Dmf_0;\mathbb{R}^d)$. The random perturbation is given via
 a standard $d$-dimensional  Brownian motion $W=(W_1,\ldots,W_d)$ defined on some probability space $(\Omega,\mathcal{F},\mathrm{P})$. The noise magnitude is given by a small parameter $\eps>0$ in front of the diffusion coefficient
$\sigma$ which is assumed to be a $\mathcal{C}^2$-smooth uniformly elliptic matrix-valued function, i.e.,
$\sigma\in\mathcal{C}^2\left(\Dmf_0;M_d(\mathbb{R})\right)$,
where $M_d(\mathbb{R})$ is the space of $d$-by-$d$ matrices with real entries and there are positive constants $\sigma_{\min},\sigma_{\max}$ such that
\[
\sigma_{\min}|\xi|^2\leq\langle \sigma(x)\xi,\xi\rangle \le \sigma_{\max}|\xi|^2 \qquad\forall
\xi\in\mathbb{R}^d,\ x\in\Dmf_0.
\]
Here $\langle\cdot,\cdot\rangle$ is the standard inner product and $|\cdot|$ is the Euclidean norm in $\R^d$. We will also use $\dist(\cdot,\cdot)$ for the Euclidean point-to-point and point-to-set distances in $\R^d$. 
Standard results on stochastic differential equations (see e.g \cite{KS1991}) imply that for any starting location $X^{\eps}(0)\in\Dmf_0$, the equation~\eqref{eq:SDE} has a unique strong solution up to
\[
\tau_{\Dmf_0}^{\eps}=\inf\{t\geq 0: X_{\eps}(t)\in\partial \Dmf_0\},
\]
the exit time from $\Dmf_0$.

Let $(S^t)$ be the flow generated by the vector field $b$, i.e., $x(t)=S^tx_0$ is the solution of the autonomous ordinary differential equation
\begin{equation}\label{eq:deterministic-ODE}
\dot{x}(t)=b(x(t)),\qquad x(0)=x_0\in\Dmf_0.
\end{equation}
This flow is defined forwards and backwards in time as long as the trajectory stays within $\Dmf_0$. 

In this paper, we are interested in the asymptotic behavior, as $\eps\downarrow 0$, of the distribution of the exit location and the exit time
\[
\tau_{\Dmf}^{\eps}=\inf\{t>0: X_{\eps}(t)\notin\Dmf\}
\]
from  a subdomain $\Dmf$ compactly contained in $\Dmf_0$. We make the following assumptions on $\Dmf$ and the vector field $b$:
\begin{enumerate}[(I)]
\item The limit set of $S^{t}$ in $\Dmf$ consists of a single point assumed to be the origin~$0$ without loss of generality.

\item For every $x\in\bar{\Dmf}\setminus\{0\}$, there is a time $T(x)$ such that $S^tx\notin \Dmf$ for $t>T(x)$ while $S^tx\in\Dmf$ for all $-\infty<t<T(x)$. Here $\bar\Dmf$ denotes the closure 
of $\Dmf$. We will denote the exit point associated with  $x$ by $\pi(x)$: 
\begin{equation}
\label{eq:deterministic_exit_points}
\pi (x) = S^{T(x)}x,\quad x\in \bar\Dmf\setminus\{0\} 
\end{equation}

\item \label{cond:linear-part-Jordan} The vector field satisfies
\[
b(x)=Ax+\psi(x)|x|^2, \qquad x\in \Dmf,
\]
where 
$\psi$ is a $\Cc^2$ vector-valued function on $\Dmf_0$, and 
$A=Db(0)$ ($D$ stands for the Jacobian matrix) is a $d$-by-$d$ matrix with one real eigenvalue $\lambda>0$ of geometric multiplicity $1$ but algebraic multiplicity $d$, i.e., it is similar to a single Jordan block
\begin{equation}\label{eq:linearized-matrix}
 \begin{bmatrix}
    \lambda & 1 & 0 & 0  & \dots & 0 \\
    0 & \lambda & 1 & 0 &\dots  & 0 \\
    \vdots & \ddots & \ddots & \ddots & \ddots & \vdots\\
    0 & 0 & \dots &  \lambda & 1 & 0\\
    0 & 0 & \dots &  0& \lambda & 1\\
    0 & 0 & \dots &0 & 0  & \lambda
\end{bmatrix}.
\end{equation}
We assume without loss of generality that $A$ is already of this form, i.e., that the generalized eigenvector basis $\{e_1,\dots, e_d\}$ coincides with the canonical basis of $\mathbb{R}^d$.
\end{enumerate}

We define $q_-, q_+\in \partial\Dmf$ to be the points such that the curve
\[
\gamma=\gamma_+\cup\gamma_-\cup\{0\},\qquad \gamma_\pm=\{S^{-t}q_\pm: t\in\mathbb{R}_+\},
\]
is $\mathcal{C}^2$-smooth and tangent to the eigenvector $e_1$ at the origin. 

\begin{enumerate}
\item[(IV)] \label{cond:transversality} We require $\partial\Dmf$ to be $\Cc^2$ in
neighborhoods of $q_-$ and $q_+$ and transversal to $\gamma$ at these points.
\end{enumerate}

We assume without loss of generality that $e_1$ points in the direction of $q_+$. The importance of these boundary points comes from the fact that the distribution $X_{\eps}(\tau_{\Dmf}^{\eps})$ is asymptotically concentrated on $\{q_-,q_+\}$. Our main result describes the joint fluctuations of the random exit location around this limit and the exit time $\tau_{\Dmf}^{\e}$.

For a vector $y\in\R^d$ , we denote  by $y^{(i)}=\langle y, e_i \rangle$ its $i$-th component in the canonical basis. 

\begin{theorem}\label{thm:main-theorem} Assuming the setting described above, 
let $X_{\eps}(0)=\eps\xi_{\eps}$, where $\xi_{\eps}$ is a family of $d$-dimensional random variables independent of $W$ and converging in probability to some random variable $\xi_0$. 

%\mblue{I think in the non-linear case, the $\pm$ is in general there due to different contributions from Proposition 3.1.}
Then on 
the same probability space there are events $A^{\pm}$, $d$-dimensional  random variables  $(\mu^{\pm}_\eps)_{\eps>0}$,
$1$-dimensional random variables $\rho$,
$\eta$, $(\theta_\eps^{\pm})_{\eps>0}$,
deterministic vectors $h_1^{\pm},h_2^{\pm}\in\R^d$,   and constants $C^{\pm}\in\R$ with the following properties:\\
on $A^\pm$,
\begin{equation}\label{eq:tau-D-e-asymp-in-theorem}
\tau_{\Dmf}^\e=\frac{1}{\lambda}\log\e^{-1}-\frac{d-1}{\lambda}\log\log\e^{-1}+\rho+C^{\pm}+\theta^{\pm}_\eps
\end{equation}
and
\begin{equation}
\label{eq:asymptotics-of-global-exit-location}
X_\eps(\tau_\Dmf^\e)=q_{\pm}
+
\left(\frac{1}{\log \eps^{-1}}
+\frac{(d-1) \log\log\e^{-1}}{\log^2\e^{-1}}
+\frac{\eta}{\log^2\e^{-1}}\right)h^{\pm}_1\\
+\frac{1}{\log^2\eps^{-1}}h^{\pm}_2+\frac{\mu_\eps^{\pm}}{\log^2\eps^{-1}};
\end{equation}

\[
\theta^{\pm}_\eps\stackrel{\Pp}{\to} 0,\quad \mu^{\pm}_\eps\stackrel{\Pp}{\to}0,\quad \eps\downarrow 0;
\]
if $d=1$, then $h_1^\pm=h_2^\pm=\mu_\eps^{\pm}=0$; If $d\ge 2$, the vector $h_1^\pm$ is tangent to $\partial \Dmf$ at $q_\pm$.
If $\partial\Dmf$ is flat (coincides with a hyperplane of codimension~$1$) in a small neighborhood of~$q_\pm$, then
$h_2^\pm$ is also tangent to~$\partial\Dmf$.

Moreover, the escape trajectory converges to the curve $\gamma$, i.e.,
\begin{equation}
\sup_{0\le t\le \tau_{\Dmf}^\eps} \dist(X_\eps(t),\gamma)\stackrel{\Pp}{\to}0,\quad \eps\downarrow 0.
\label{eq:exit_happens_along_gamma} 
\end{equation}

\end{theorem}

{\noindent \bf Remarks:}
\begin{enumerate}
%\red{First, I wrote this: \\ In fact, the choice of $(A_\eps^\pm)_{\eps>0}$ is far from unique. The theorem
%holds for every family $(A_\eps^\pm)_{\eps>0}$ satisfying~\eqref{eq:condition_on_A_eps_pm} since
%it is easily seen that  if it holds for some $(A_\eps^\pm)_{\eps>0}$, then it also holds for any other family of events
%$(\tilde A^\pm_\eps)_{\eps>0}$ that satisfies $\Pp(A^\pm_\eps\triangle \tilde A^{\pm}_\eps)\to0$ as $\eps\to 0$. \\
%Now I think we should get rid of $A^{\pm}_\eps$ totally. Everything can be stated in terms of
%the limiting variable $\nu=\chi^{(d)}$:
%\\
\item Precise expressions for the random variables involved in the statement of this theorem will
be given in the course of the proof and in the auxiliary statements that we will invoke.
The events $A^\pm$ in this theorem are defined by $A^\pm=\{\sgn\chi^{(d)}=\pm 1\}=\{\pm \chi^{(d)}>0\}$. Here~$\chi$ is a random vector responsible for the asymptotic 
direction of exit 
  introduced in the main auxiliary Theorem~\ref{thm:linear-result}, see~\eqref{eq:eta-and-chi}. It is defined in  \eqref{eq:introducing_chi} 
  in terms of the ingrediends in the variation of constants formula (the initial condition and
  the contribution from noise)
for an auxiliary equation defined in
 \eqref{eq:definitions_of_N_and_D}. The random variable $\eta$ is also defined in~\eqref{eq:eta-and-chi}, in the statement of Theorem~\ref{thm:linear-result}.
\item The random variable $\rho$ serves both directions of exit. The only difference
between the two directions in the asymptotic behavior of the exit time  is encoded in constants $C^{\pm}$. In fact,~$\rho$~is defined up to an additive shift that has to be compensated by adjusting~$C^{\pm}$. One can achieve uniqueness of $\rho$ and $C^{\pm}$ by requiring $\E\rho=0$. 
\item
As it will be clear from the proof, the direction of exit and the scaling limit of the exit distribution are asymptotically determined by the noise picked up in an infinitesimal neighborhood of the origin in the direction of $e_{d-1}$ and $e_d$.
\item
It will become clear that in some situations we can, in fact, provide more detailed information than 
Theorem~\ref{thm:main-theorem}. A nice formulation is possible, for example, in the linear case, see 
\eqref{eq:precise-formula-in-all-directions}.
\item Although it is possible to consider more general scalings 
$X_\eps(0)=\eps^\alpha \xi_\eps$ for a convergent family $(\xi_\eps)_{\e\geq 0}$ and an arbitrary scaling exponent $\alpha>0$, it will be clear from our analysis that 
the case $\alpha=1$ considered in Theorem~\ref{thm:main-theorem} is the most interesting one.
In fact, if $\alpha<1$, then the noise is asymptotically negligible and the behavior is dominated by the deterministic dynamics, and the case where $\alpha>1$ effectively reduces to $\alpha=1$ since
$\eps^\alpha \xi_\eps=\eps\cdot\eps^{\alpha-1}\xi_\eps$ and $\eps^{\alpha-1}\xi_\eps\stackrel{\mathrm{P}}{\to}0$, so the  influence of the initial condition asymptotically vanishes.
\item
In~\cite{Bak2010} and~\cite{Bak2011}, the results
had to be stated in terms of convergence in distribution since
the contributions from the stable directions were of the leading order of magnitude and converged only in distribution. In the setting of the present paper, in the absence of stable directions, we are able to state the results in terms of convergence in probability. However, our result holds
even when any nonleading, positive or negative,
eigenvalues are present, and a smooth conjugation to linear dynamics exists. In this case, the contributions from nonleading eigendirections
 are of smaller order than
the scales relevant for the asymptotics in~\eqref{eq:asymptotics-of-global-exit-location}.
\item One can restate the theorem for the situation where only convergence in distribution is required for the initial condition, and use Skorokhod's representation theorem on realization of weak convergence by almost sure convergence. 
\item
Let us emphasize the connection to the existing results. It was shown for a more general
setting in \cite{Eizenberg:MR749377} that the marginal distribution of $X_{\eps}\left(\tau_{\Dmf}^{\eps}\right)$ asymptotically concentrates on $\{q_+,q_-\}$. The precise asymptotics of the marginal limiting law of $\tau_{\Dmf}^{\eps}$ was computed for linear drift in~\cite{Buterakos}. The main novelty in our result  is the expansion \eqref{eq:asymptotics-of-global-exit-location} providing
a precise asymptotic description of fluctuations of the random exit point around $q_\pm$, along with joint asymptotics for the fluctuations of the exit time.
\end{enumerate}

\section{Proof of Theorem \ref{thm:main-theorem}}
\label{sec:proof-of-main}
Our approach is based on two steps: (i) studying the system in a small neighborhood of the origin where a change of coordinates conjugates the dynamics to a linear system; (ii) describing the behavior of $X_{\eps}$ as it follows the curve $\gamma$ between the linearizable neighborhood and the exit points~$q_\pm$.

We start with the first part. It was demonstrated in \cite{Eizenberg:MR749377}, that 
under condition \eqref{cond:linear-part-Jordan}, there is a neighborhood $\Umf$ of the origin and a smooth diffeomorphism $f:\Umf\mapsto\mathbb{R}^n$ given by
\begin{equation}
\label{eq:def-of-f}
f(x)=\lim_{t\to\infty} e^{At} S^{-t}x=x-\int_0^\infty e^{As}\psi(S^{-s}x)|S^{-s}x|^2 ds,
\end{equation}
with inverse $g$ that conjugates the linear and non-linear dynamics, i.e. 
\begin{equation}\label{eq:conjugation-relation}
f(S^tx)=e^{At}f(x)\qquad\textrm{or}\qquad Df(x)b(x)=Af(x).
\end{equation}
The integral term in~\eqref{eq:def-of-f} 
is quadratic to the leading order in small $x$, which implies
\begin{equation}
f(0)=0,\qquad Df(0)=I,
\label{eq:Df(0)}
\end{equation}
where $I$ is the identity matrix.

When $X_{\eps}(0)\in \Umf$, let $\tau_{\Umf}^{\eps}$ be the first time when $X_{\eps}(t)$ exits $\Umf$. If we set $Y_{\eps}(t)=f(X_{\eps}(t))$, then It\^o's formula and \eqref{eq:conjugation-relation} imply that this process satisfies the stochastic differential equation
\begin{equation}\label{eq:linear-SDE}
dY_{\eps}(t)=AY_{\eps}(t)dt+\eps\tilde{\sigma}\left(Y_{\eps}(t)\right)dW(t)+\frac{\eps^2}{2} L(Y_{\eps}(t))dt,
\end{equation}
for $t<\tau_{\Umf}^{\eps}$, where 
\[
\tilde{\sigma}(y)=Df\left(g(y)\right)\sigma(g(y)),\qquad L_i(y)=\sum_{j,l=1}^d\partial_j\partial_l f_i(g(y))a_{jl}\left(g(y)\right),\quad i=1,\dots,d,
\]
and $a(x)=(\sigma\sigma^T)(x)$.

  We denote $\|y\|_{\infty}=\max\{|y^{(k)}|:\ k=1,\dots,d\}$. We are going to study the precise asymptotics of the exit time and location from the box 
\[
\Vmf=\{\|y\|_{\infty}\leq R\},
\]
where $R$ is chosen small enough such that $g(\Vmf)\subset\Umf$. 
Namely, the following theorem, proved in Section \ref{sec:proof-of-linear}, characterizes the joint scaling behavior, as $\e\downarrow 0$, of the stopping time
\[
 \tau_{\mathfrak{B}}^\eps = \inf\{t>0:\ \|Y_\eps(t)\|_{\infty}=R\}
\]
and the exit point $Y_{\eps}(\tau_{\mathfrak{B}}^{\eps})$. 

\begin{theorem}\label{thm:linear-result}
Let $Y_{\eps}(0)=\e\tilde{\xi}_{\e}$, 
where $\tilde{\xi}_{\eps}$ is a family of $d$-dimensional random variables independent of $W$ and converging in probability to some $\tilde{\xi}_0$. 
Then, on 
the same probability space, there are $d$-dimensional  random variables $N$, $\chi$, $(\zeta_\eps)_{\eps>0}$,
$1$-dimensional random variables
$\rho, \eta$, $(\theta_\eps^\pm)_{\eps>0}$ with the following properties: 

\begin{equation}\label{eq:tau-B-e-asymp-in-theorem}
\tau_{\Vmf}^\e=\frac{1}{\lambda}\log\e^{-1}-\frac{d-1}{\lambda}\log\log\e^{-1}+\rho+\theta_\eps^{\pm};
\end{equation}
\[
\rho=-\frac{1}{\lambda}\log\frac{|\chi^{(d)}|}{R(d-1)!\lambda^{d-1}};
\]
on $A^\pm=\{\pm \chi^{(d)}>0\}$,
\begin{multline}
\label{eq:asymptotics-of-exit-from-small}
Y_\e(\tau_\Vmf^\e)=\pm R\Biggl[e_1+
\lambda(d-1)\left(\frac{1}{\log \eps^{-1}}
+\frac{(d-1) \log\log\e^{-1}}{\log^2\e^{-1}}
+\frac{\eta}{\log^2\e^{-1}}\right)e_2
\\+\frac{\lambda^2(d-1)(d-2)}{\log^2\eps^{-1}}e_3+\frac{\zeta_\eps}{\log^2\eps^{-1}}\Bigg];
\end{multline}
\[
\theta^\pm_\eps\stackrel{\Pp}{\to} 0,\quad \zeta_\eps\stackrel{\Pp}{\to}0,\quad \eps\downarrow 0;
\]
\[
\langle\zeta_\eps,e_1\rangle=0;
\]
\begin{equation}
\label{eq:eta-and-chi}
\chi=\tilde{\xi}_0+N,\quad\eta=-\lambda\frac{\chi^{(d-1)}}{\chi^{(d)}}+\log\frac{|\chi^{(d)}|}{R(d-1)!\lambda^{d-1}};
\end{equation}
$N$ is independent of $\tilde \xi_0$, it is centered Gaussian, with covariance matrix given by
\begin{equation}\label{eq:limit-cov-small-box}
\mathrm{E}N^{(i)}N^{(j)}=\sum_{p=0}^{d-i}\sum_{q=0}^{d-j}\binom{p+q}{q}
\frac{(-1)^{p+q}}{(2\lambda)^{p+q+1}}a_{p+i,q+j}(0);
\end{equation} 
Also,
\begin{equation}
\label{eq:motion-is-close-to-axis}
\sup_{t\leq \tau_\Vmf^\e}\mathrm{dist}\big(Y_\e(t),\mathrm{span}(e_1)\big)\stackrel{\Pp}{\to} 0,\quad \eps \downarrow 0.
\end{equation}
\end{theorem}
\begin{remark}\rm The term containing $e_3$  in~\eqref{eq:asymptotics-of-exit-from-small} is not present for $d=1,2$. The term containing~$e_2$ in~\eqref{eq:asymptotics-of-exit-from-small} does not appear for $d=1$. This may be formally
achieved by setting $e_i=0$ for $i>d$ and also can be formally seen from the presence of factors $(d-1)$
and $(d-2)$ in front of these terms. In fact, in the case $d=1$, 
the identity~\eqref{eq:asymptotics-of-exit-from-small} is trivial, and the identity~\eqref{eq:tau-B-e-asymp-in-theorem} is contained in~\cite{Day95}.

\end{remark}
\begin{remark}\rm Only components $N^{(d-1)}$, $N^{(d)}$ are effectively used in the theorem, but 
it is convenient to introduce all $d$ coordinates to be used in the proof.
\end{remark}
\begin{remark}\rm The theorem implies that the asymptotic choice of the outgoing direction is described by
\[
\mathrm{P}\left\{Y_{\eps}^{(1)}(\tau_{\Vmf}^{\e})=\pm R\right\}\to \mathrm{P}\left\{\pm \chi^{(d)}>0\right\},\quad \eps \downarrow 0.
\] 
\end{remark}

\begin{corollary} \label{cor:linear-result-to-X} There are deterministic vectors $u^{\pm}_1,u^{\pm}_2\in\R^d$,
and a family of random vectors $(\beta^{\pm}_\eps)_{\eps>0}$ such that
$\beta_\eps^{\pm}\stackrel{\Pp}{\to}0$ and,
on events $A^{\pm}$ introduced in Theorem~\ref{thm:linear-result},
\begin{multline}
\label{eq:asymptotics-of-exit-from-small-transformed}
X_\eps(\tau_\Vmf^\e)=g(Y_\e(\tau_\Vmf^\e))=g(\pm Re_1)
+
\left(\frac{1}{\log \eps^{-1}}
+\frac{(d-1) \log\log\e^{-1}}{\log^2\e^{-1}}
+\frac{\eta}{\log^2\e^{-1}}\right)u^{\pm}_1\\
+\frac{1}{\log^2\eps^{-1}}u^{\pm}_2+\frac{\beta^{\pm}_\eps}{\log^2\eps^{-1}}.
\end{multline}
If $d=1$, then $u^{\pm}_1=u^{\pm}_2=\beta_\eps^\pm=0$. If $d=2$, then $u^{\pm}_1$ and $u^{\pm}_2$ are tangent to $\partial \Vmf$ at $g(\pm Re_1)$ and collinear to each other.

Also,
\begin{equation}
\label{eq:tracking_gamma_in_small_neighb}
\sup_{0\le t\le \tau_{\Vmf}^\eps} \dist(X_\eps(t),\gamma)\stackrel{\Pp}{\to}0,\quad \eps\downarrow 0.
\end{equation}
\end{corollary}
\begin{proof} The expansion \eqref{eq:asymptotics-of-exit-from-small-transformed} follows directly from~\eqref{eq:asymptotics-of-exit-from-small}
and the Taylor expansion for $g$ near~$\pm Re_1$ if we take into account that the only
nonnegligible nonlinear contributions  come from quadratic terms and appear 
in the form of $1/\log^2\eps^{-1}$ with coefficients given by second partial derivatives of $g$
at $\pm Re_1$. 
\end{proof}

\bigskip

\noindent\textit{Proof of Theorem \ref{thm:main-theorem}.~}
This proof is based on~Theorem~\ref{thm:linear-result}, its Corollary~\eqref{cor:linear-result-to-X}, the strong Markov property, and a simple analysis of  
the evolution of  $X_{\e}$ along $\gamma$ between leaving $g(\Vmf)$ and exiting~$\Dmf$, which is dominated by the deterministic dynamics applied to the initial condition $X_\eps(\tau_{\Vmf}^{\e})$
since the noise contributions are much smaller. 

The restriction of the map $\pi$ defined in~\eqref{eq:deterministic_exit_points}
to a relative neighborhood $U\subset g(\partial \Vmf)$ of $g(\pm Re_1)$ is the Poincar\'e map for the flow $(S^t)$ 
between the surfaces $U$ and $\partial\Dmf$.  Due to our smoothness assumptions and the transversality condition~(IV), 
$\pi$ is $\Cc^2$ in~$U$. So if $X_\eps(\tau_\Vmf^\e)\in U$, then applying the second order Taylor expansion of $\pi$
near $g(\pm Re_1)$ to $X_\eps(\tau_\Vmf^\e)$ \eqref{eq:asymptotics-of-exit-from-small-transformed} gives
\begin{multline}
\label{eq:asymptotics-under-Poincare-map}
\pi\left(X_\eps(\tau_\Vmf^\e)\right)=q_{\pm}
+
\left(\frac{1}{\log \eps^{-1}}
+\frac{(d-1) \log\log\e^{-1}}{\log^2\e^{-1}}
+\frac{\eta}{\log^2\e^{-1}}\right)h^{\pm}_1
\\
+\frac{1}{\log^2\eps^{-1}}h^{\pm}_2+\frac{\tilde \beta_\eps^{\pm}}{\log^2\eps^{-1}}
\end{multline}
for some $\tilde \beta_\eps^{\pm}\stackrel{\Pp}{\to}0$. Namely, $h_1^\pm$ is linear in $u^\pm_1$:
\[h_1^\pm = D \pi (g(\pm Re_1))u^\pm_1,\] whereas
$h_2^\pm$ is composed of the linear part $D \pi (g(\pm Re_1))u^\pm_2$ and a quadratic form in $u^\pm_1$:
\[
 {h_2^\pm}^{(k)}=\sum_{i} \frac{\partial}{\partial x^{(i)}} \pi^{(k)} (g(\pm Re_1)){u^\pm_2}^{(i)}
+\frac{1}{2}\sum_{i,j} \frac{\partial^2}{\partial x^{(i)}\partial x^{(j)}} \pi^{(k)} (g(\pm Re_1)){u^\pm_1}^{(i)}{u^\pm_1}^{(j)}. 
 \]
The classical expansion of solutions in powers of small $\eps$ on finite time intervals,  see~\cite[Chapter 2]{FW2012}, 
implies that for any $T$,
\begin{equation*}
%\label{eq:tracking_deterministic_orbit}
 \Pp\left\{\sup_{t\in [0,T]} \left|X_\eps(\tau^\eps_\Vmf+t) - S^t g(X_\eps(\tau^\eps_\Vmf))\right|>\eps^{1/2}\right\}\to 0,\quad \eps\to 0.
\end{equation*}
Since $S^t g(X_\eps(\tau^\eps_\Vmf))$ is itself close to $\gamma$, we 
can use the transversality condition 
to  obtain~\eqref{eq:exit_happens_along_gamma} from~\eqref{eq:tracking_gamma_in_small_neighb} and to derive~\eqref{eq:tau-D-e-asymp-in-theorem} (with $C^\pm=T(g(\pm Re_1))$) from \eqref{eq:tau-B-e-asymp-in-theorem}. We also get
\[
\Pp\{|X_\eps(\tau_{\Dmf}^{\e})-\pi(X_\eps(\tau_\Vmf^\e))|>\eps^{1/2}\}\to 0, \quad \eps\to 0.
\]
Combining this with~\eqref{eq:asymptotics-under-Poincare-map}, we obtain~\eqref{eq:asymptotics-of-global-exit-location}
and complete the proof of Theorem~\ref{thm:main-theorem}.
\qed

\section{The linear system}\label{sec:proof-of-linear}
In this section, we prove Theorem \ref{thm:linear-result}. We will repeatedly make use of the elementary formulas
\begin{equation}\label{eq:elementary-formula}
(1+x)^p=1+px+\mathcal{O}(x^2),\qquad \log(1+x)=x+\mathcal{O}(x^2),\quad x\to 0.
\end{equation}

Duhamel's formula for the SDE \eqref{eq:linear-SDE} implies
\begin{equation}\label{eq:Duhamel}
Y_{\eps}(t)=\eps e^{At}\left[\tilde\xi_{\eps}+ N_{\eps}(t)+\eps D_{\eps}(t)\right],\quad t\le \tau_{\Umf}^\eps,
\end{equation}
where
%\red{[$a,\sigma,\tilde\sigma,L$ are defined only in some neighborhood, so don't we need some localization procedure here? Or should we extend them to the entire $\R^d$ in some way]} 
\begin{equation}
\label{eq:definitions_of_N_and_D}
N_{\eps}(t)=\int_0^t e^{-As}\tilde{\sigma}(Y_{\eps}(s))dW(s),\qquad D_{\eps}(t)=\frac{1}{2}\int_0^t e^{-As} L(Y_{\eps}(s))ds.
\end{equation}
Recall that
\[
e^{At}=e^{\lambda t}\begin{bmatrix}
    1 & t & \frac{ t^2}{2!} & \frac{ t^3}{3!}  & \dots & \frac{ t^{d-1}}{(d-1)!} \\
    0 & 1 &  t & \frac{t^2}{2!} &\dots  &  \frac{ t^{d-2}}{(d-2)!}  \\
    \vdots & \ddots & \ddots & \ddots & \ddots & \vdots\\
    0 & 0 & \dots &  1 &  t & \frac{t^2}{2!}\\
    0 & 0 & \dots &  0& 1 &  t\\
    0 & 0 & \dots &0 & 0  & 1 
\end{bmatrix},\qquad 
e^{-As}=e^{-\lambda s}\begin{bmatrix}
    1 & - s & \frac{(-s)^2}{2!} & \dots & \frac{(- s)^{d-1}}{(d-1)!} \\
    0 & 1 &  - s &\dots  &  \frac{(- s)^{d-2}}{(d-2)!}  \\
    \vdots & \ddots  & \ddots & \ddots & \vdots\\
    0 & 0 & \dots & 1 & - s\\
    0 & 0 & \dots  & 0  & 1 
\end{bmatrix},
\]
so for any vector $\xi\in\mathbb{R}^d$, we have
\begin{equation}\label{eq:multipl-with-exp-coordinatewise}
(e^{At}\xi)^{(i)}=e^{\lambda t}\sum_{j=0}^{d-i}\frac{ t^j}{j!}\xi^{(i+j)},\qquad (e^{-As}\xi)^{(i)}=e^{-\lambda s}\sum_{j=0}^{d-i}\frac{(- s)^j}{j!}\xi^{(i+j)}.
\end{equation}
In particular,
\begin{equation}\label{eq:sum-form-N}
N_{\eps}^{(i)}(t)=\int_0^te^{-\lambda s}\sum_{k=1}^d\sum_{j=0}^{d-i}\frac{(-s)^j}{j!}\tilde{\sigma}_{i+j,k}(Y_{\eps}(s))\,dW_k(s),
\end{equation}
and 
\begin{equation}\label{eq:sum-form-D}
D_{\eps}^{(i)}(t)=\frac{1}{2}\int_0^te^{-\lambda s}\sum_{j=0}^{d-i}\frac{(- s)^j}{j!}L_{i+j}(Y_{\eps}(s))\,ds.
\end{equation}
The following lemma implies that $N_{\e}(t)$ is of the order of one for all times while $D_{\e}(t)$ is bounded.

\begin{lemma}\label{lem:a-priori-N-D}
There are constants $c,D_0>0$ such that
\begin{equation}\label{eq:bounded-D}
\sup_{\e\geq 0}\mathrm{P}\left\{\sup_{t\leq\tau_{\Vmf}^{\e}}\|N_{\e}(t)\|_{\infty}>z\right\}\leq\frac{c}{z^2}
,\qquad\sup_{\eps>0,\ t\leq\tau_{\Vmf}^{\e}}\|D_{\eps}(t)\|_{\infty}\leq D_0.
\end{equation}
\end{lemma}

\begin{proof}The second claim follows from $\|e^{-As}\|_{\infty}\leq Cs^{d-1}e^{-\lambda s}$ and the boundedness of $L$ (which is due to the boundedness of $\sigma$, $g$, and $f$ together with its derivatives):
\[
\sup_{t\leq\tau_{\Vmf}^{\e}}\|D_{\eps}(t)\|_{\infty}\leq C\sup_{t\geq 0}\int_0^te^{-\lambda s}s^{d-1}ds<\infty.
\]
To prove the first claim, observe (from, e.g., \eqref{eq:sum-form-N}) that each component $N_{\eps}^{(i)}(t)$ is a martingale. Thus we write using Chebyshev's inequality and the BDG inequality
\begin{multline*}
\mathrm{P}\left\{\sup_{t\leq\tau_{\Vmf}^{\e}}\|N_{\eps}(t)\|_{\infty}>z\right\}\leq 
d\max_{i=1,\dots, d}\mathrm{P}\left\{\sup_{t\leq\tau_{\Vmf}^{\e}}|N_{\eps}^{(i)}(t)|>z\right\}\\
\leq
\frac{d\cdot\max_{i=1,\dots,d}\mathrm{E}\sup_{t\leq \tau_{\Vmf}^{\e}}|N_{\eps}^{(i)}(t)|^2}{z^2}\leq C\frac{\max_{i=1,\dots,d}\sup_{t\geq 0}\mathrm{E}\langle N_{\eps}^{(i)}\rangle_{t\wedge\tau_{\Vmf}^{\e}}}{z^2},
\end{multline*}
where $\langle\cdot\rangle$ is the quadratic variation process. The right-hand side is uniformly bounded 
due to the argument that we have used above for $D_{\e}$.
\end{proof}

We are going to study the precise asymptotics of the exit time and location from the box $\Vmf$. We do this in two steps: (1) in a small $\eps$-dependent neighborhood
\[
\Vmf_{\e}=\{\|y\|_{\infty}\leq \e^{\alpha}\}
\]
 of the origin, where $\alpha\in(0,1)$ so  $\Vmf_{\e}$ is  still larger than the noise magnitude; (2) between 
 exiting~$\Vmf_{\eps}$ and the final exit from $\Vmf$. 

In part (1),  $Y_{\e}$ is close to the origin, which allows us to control the error of the linear approximation and to approximate $N_{\e}(t)$ determining the exit direction by a Gaussian random vector. In part (2), the deterministic process dominates, and and we can control the deviations of $Y_{\e}$ from the corresponding solution of \eqref{eq:deterministic-ODE}.

\subsection{Exit from a small neighborhood of the origin}
Let $\tBe$ be the exit time of the process $Y_\eps$ from $\Vmf_{\e}$. 
Our assumption on $X_{\e}(0)$ implies
\[
\lim_{\e\downarrow 0}\mathrm{P}\{Y_{\e}(0)\in\Vmf_{\e}\}=1.
\]

\begin{lemma}\label{lem:small-box-exit-large}
The exit time $\tBe$ converges to infinity in probability, i.e., for all $T\ge 0$,
\[
\lim_{\e\downarrow 0}\mathrm{P}\{\tBe\leq T\}=0.
\]
\end{lemma} 

\begin{proof}
Using \eqref{eq:Duhamel}, we can write
\begin{multline*}
\mathrm{P}\{\tBe\leq T\}\leq
 \mathrm{P}\left\{C \e e^{\lambda T} T^{d-1}\|\tilde\xi_{\e}\|_{\infty}\geq \frac{\e^\alpha}{4}\right\}
 + 
\mathrm{P}\left\{C \e e^{\lambda T} T^{d-1}\sup_{t\leq T\wedge\tau_{\Vmf}^{\eps}}\|N_{\eps}(t)\|_{\infty}\geq\frac{\e^{\alpha}}{4}\right\}\\ +\mathrm{P}\left\{C \e^2e^{\lambda T} T^{d-1}\sup_{t\leq T\wedge\tau_{\Vmf}^{\eps}}\|D_{\eps}(t)\|_{\infty}\geq\frac{\e^{\alpha}}{4}\right\},
\end{multline*}
where we used $\sup_{t\in[0,T]}\|e^{At}\|_{\infty}\leq C e^{\lambda T} T^{d-1}$. The first term converges to zero by the tightness of $\tilde\xi_{\e}$, while the second and third terms do so by Lemma \ref{lem:a-priori-N-D}.
\end{proof}

\begin{lemma}\label{lem:N-limit-small-box} As $\epsilon\downarrow 0$,
$N_{\e}(\tBe)$ converges in probability to
a centered Gaussian vector N, independent of $\tilde\xi_0$, with covariance matrix described in \eqref{eq:limit-cov-small-box}.
\end{lemma}

\begin{proof}
Let us consider the Gaussian martingale
\[
M(t)=\int_0^te^{-As}\sigma(0)dW(s),
\]
with quadratic variation matrix
\[
\langle M\rangle_t=\int_0^te^{-As}a(0)e^{-A^Ts}ds.
\]
This matrix is uniformly bounded in $t$ and therefore the martingale convergence theorem implies the existence of the almost sure, componentwise limit
\[
N=\int_0^\infty e^{-As}\sigma(0)dW(s)=\lim_{t\to\infty}M(t),
\]
a centered Gaussian vector with covariance matrix that can be computed using \eqref{eq:multipl-with-exp-coordinatewise}:
\[
\mathrm{E}N^{(i)}N^{(j)}=\sum_{p=0}^{d-i}\sum_{q=0}^{d-j}\frac{(-1)^{p+q}}{p!q!}a_{p+i,q+j}(0)\int_0^\infty e^{-2\lambda s}s^{p+q}ds.
\]
 Using this and $\int_0^{\infty}x^n e^{-ax}dx=n!/a^{n+1}$, we derive~\eqref{eq:limit-cov-small-box}. 

A straightforward calculation based on the BDG inequality, the Lipschitzness of $\tilde{\sigma}(\cdot)$, the identity $Df(0)=I$, and the definition of $\tBe$ shows that
\[
\sup_{t\leq \tBe}\|N_{\e}(t)-M(t)\|_{\infty}\stackrel{\Pp}{\to} 0.
\]
This, together with Lemma \ref{lem:small-box-exit-large}, finishes the proof.
\end{proof}

Let us now introduce the exit time in the individual directions
\[
\tau_{i}^{\eps}=\inf\{t>0: |Y_{\eps}^{(i)}(t)|=\e^{\alpha}\},
\]
where we recall that $Y_{\eps}^{(i)}$ is the $i$th component of $Y_{\eps}$. Clearly, $\tBe=\min_i \tau_{i}^{\eps}$.

\begin{lemma}\label{lem:typical-exit-side}
The exit happens in the direction of $e_1$ with overwhelming probability, i.e.,
\[
\lim_{\e\downarrow 0}\mathrm{P}\left\{\tBe=\tau_{1}^{\e}\right\}=1.
\]
\end{lemma}

\begin{proof}
Observe that \eqref{eq:Duhamel}, \eqref{eq:multipl-with-exp-coordinatewise}, and \eqref{eq:sum-form-N} combined with Lemmas \ref{lem:small-box-exit-large}--\ref{lem:N-limit-small-box} imply the first order approximation
\begin{equation}\label{eq:Y-i-formulas-small-box}
Y_{\e}^{(i)}(\tBe)=\e e^{\lambda \tBe}\frac{\tBe^{d-i}}{(d-i)!}\left[\tilde\xi_{\e}^{(d)}+N_{\e}^{(d)}(\tBe)\right]\left(1+o_{\Pp}(1)\right),\quad i=1,\dots,d,
\end{equation}
where we write $A_\eps=o_{\Pp}(B_\eps)$ if $A_\eps/B_\eps \stackrel{\Pp}{\to} 0$ as $\eps\downarrow 0$.
 This means, by Lemma \ref{lem:small-box-exit-large}, that 
\[
\frac{Y_{\eps}^{(i)}(\tBe)}{Y_{\eps}^{(1)}(\tBe)}\stackrel{\Pp}{\to} 0,
\quad \eps\downarrow 0,
\qquad i=2,\dots,d,
\]
which proves the claim since
\[
\mathrm{P}\left\{\tBe\neq\tau_1^{\e}\right\} \leq\sum_{i=1}^d\mathrm{P}\left\{\left|\frac{Y_{\eps}^{(i)}(\tBe)}{Y_{\eps}^{(1)}(\tBe)}\right|\geq 1\right\}\to 0.
\]
\end{proof}

Let us introduce the abbreviations
\begin{equation}
\label{eq:introducing_chi}
\chi_{\e}(t)=\tilde\xi_{\e}+N_{\e}(t)+\eps D_{\e}(t),\qquad \chi_{\e}=\chi_{\e}(\tBe),\qquad \chi=\tilde\xi_0+N,
\end{equation}
and notice that
\begin{equation}
\label{eq:chi-converges}
 \chi_{\e}\stackrel{\mathrm{\Pp}}{\to}\chi
\end{equation}
due to Lemma \ref{lem:N-limit-small-box} and \eqref{eq:bounded-D}. We can now formulate the main result of the subsection providing several leading terms in an expansion for $\tBe$.
To avoid lengthy formulas, we introduce the random variables 
\[
G_{\e}(\alpha)=\log\left(|\chi_{\e}^{(d)}|\frac{(1-\alpha)^{d-1}}{(d-1)!\lambda^{d-1}}\right),\qquad \eta_\e(\alpha)=-\lambda\frac{\chi_\e^{(d-1)}}{\chi_\e^{(d)}}+G_\eps(\alpha),
\]
where $\chi_\e^{(d-1)}=0$ for $d=1$.

\begin{proposition}\label{prop:tau-B-eps-asymp} The following representation holds:
\begin{equation}\label{eq:asymp-exit-time}
\tBe=\frac{1-\alpha}{\lambda}\log\e^{-1}-\frac{d-1}{\lambda}\log\log\e^{-1}-\frac{1}{\lambda}G_\eps(\alpha)+\frac{ \tilde{K}(\e)}{\lambda},
\end{equation}
where
\begin{equation}\label{eq:e-K-eps}
\tilde{K}(\e)=\frac{(d-1)^2}{1-\alpha}\frac{\log\log\e^{-1}}{\log\e^{-1}}+
\frac{d-1}{(1-\alpha)\log\e^{-1}}\eta_\e(\alpha)+o_{\Pp}\left(\frac{1}{\log\e^{-1}}\right),\quad \e\downarrow 0.
\end{equation}
In particular, 
\begin{equation}
\label{eq:K-eps-o1}
\tilde{K}(\e)=o_{\Pp}(1).
\end{equation}
\end{proposition}

\begin{proof} Using \eqref{eq:Duhamel}, \eqref{eq:multipl-with-exp-coordinatewise}, \eqref{eq:sum-form-N}, Lemmas \ref{lem:small-box-exit-large}--\ref{lem:N-limit-small-box}, and keeping 
one more term compared to~\eqref{eq:Y-i-formulas-small-box}, we obtain
\[
Y_{\e}^{(1)}(\tBe)=\e e^{\lambda\tBe}\frac{\tBe^{d-1}}{(d-1)!}\chi_{\e}^{(d)}\left(1+\frac{(d-1)}{\tBe}\frac{\chi_{\e}^{(d-1)}}{\chi_{\e}^{(d)}}+\mathcal{O}_{\Pp}\left(\tBe^{-2}\right)\right),
\]
where $A_\eps=\mathcal{O}_{\Pp}(B_\eps)$ means that the distributions of $A_\eps/B_\eps$ form a tight family for small $\eps$.
On $\{\tBe^{\e}=\tau^{\eps}_1\}$, which has probability converging to one by Lemma \ref{lem:typical-exit-side}, we have $|Y_{\eps}^{(1)}(\tBe)|=\e^{\alpha}$ and consequently $\tBe$ is a solution of the equation
\begin{equation}\label{eq:exit-time-eq}
 \e^{\alpha}=\e e^{\lambda\tBe}\frac{\tBe^{d-1}}{(d-1)!}|\chi_{\e}^{(d)}|\left(1+\frac{(d-1)}{\tBe}\frac{\chi_{\e}^{(d-1)}}{\chi_{\e}^{(d)}}+\mathcal{O}_{\Pp}\left(\tBe^{-2}\right)\right).
\end{equation}
We now define $\tilde K(\eps)$ by \eqref{eq:asymp-exit-time}, or, equivalently, by
\begin{equation}
\label{eq:exp-of-lambda-tau}
 e^{\lambda \tBe}=\frac{\e^{-(1-\alpha)}e^{\tilde{K}(\e)}}{(\log\e^{-1})^{d-1}|\chi_{\e}^{(d)}|}\frac{(d-1)!\lambda^{d-1}}{(1-\alpha)^{d-1}},
\end{equation}
Plugging this into \eqref{eq:exit-time-eq}, we obtain
\begin{multline}
\label{eq:solving-for-K-0}
1=e^{\tilde{K}(\e)}
\left[1-\frac{d-1}{1-\alpha}\frac{\log\log\e^{-1}}{\log\e^{-1}}-\frac{G_\eps(\alpha)}{(1-\alpha)\log\e^{-1}}+\frac{\tilde K(\eps)}{(1-\alpha)\log\e^{-1}}\right]^{d-1}
\times \\
\times
\left[1+\frac{(d-1)}{\tilde \tau_\eps}\frac{\chi_{\e}^{(d-1)}}{\chi_\e^{(d)}}+o_{\Pp}\left(\frac{1}{\log\e^{-1}}\right)\right].
\end{multline}
Due to~\eqref{eq:chi-converges}, this implies~\eqref{eq:K-eps-o1}. By~\eqref{eq:K-eps-o1}  and~\eqref{eq:asymp-exit-time},
\begin{equation}
\label{eq:asymp-for-tilde-tau}
\tBe=\frac{1-\alpha}{\lambda}\log\e^{-1}\left(1+o_{\Pp}(1)\right).
\end{equation}
Substituting \eqref{eq:K-eps-o1} and~\eqref{eq:asymp-for-tilde-tau}  into  \eqref{eq:solving-for-K-0} and using~\eqref{eq:chi-converges}, we get
\begin{multline}
\label{eq:solving-for-K-1}
1=e^{\tilde{K}(\e)}\left[1-\frac{d-1}{1-\alpha}\frac{\log\log\e^{-1}}{\log\e^{-1}}-\frac{G_\eps(\alpha)}{(1-\alpha)\log\e^{-1}}+o_{\Pp}\left(\frac{1}{\log\e^{-1}}\right)\right]^{d-1}\times \\
\times
\left[1+\frac{(d-1)\lambda}{(1-\alpha)\log\e^{-1}}\frac{\chi_{\e}^{(d-1)}}{\chi_\e^{(d)}}+o_{\Pp}\left(\frac{1}{\log\e^{-1}}\right)\right],
\end{multline}
which can be written, using \eqref{eq:elementary-formula}, as
\begin{equation}
\label{eq:solving-for-K-2}
1=e^{\tilde{K}(\e)}\left[1-\frac{(d-1)^2}{1-\alpha}\frac{\log\log\e^{-1}}{\log\e^{-1}}-\frac{d-1}{(1-\alpha)\log\e^{-1}}\eta_\e(\alpha)+o_{\Pp}\left(\frac{1}{\log\e^{-1}}\right)\right].
\end{equation}
Combining this with the other formula in \eqref{eq:elementary-formula} implies \eqref{eq:e-K-eps}
thus completing the proof.
\end{proof}

Finally, we describe the asymptotic exit location from $\Vmf_{\e}$ in terms of $\tBe$. Of course, we could use the asymptotics obtained in Proposition \ref{prop:tau-B-eps-asymp}. However, we are ultimately interested in the exit distribution from $\Vmf$ and our choice turns out to be convenient when we combine the following result with the dynamics between leaving $\Vmf_\e$ and leaving $\Vmf$ in the next subsection.
\begin{proposition}\label{prop:Y-eps-asymp}
For $i=1,\dots, d$, 
\[
Y_{\eps}^{(i)}(\tBe)=\frac{\e^{\alpha}\sgn(\chi_\e^{(d)})}{\tBe^{i-1}}\frac{(d-1)!}{(d-i)!}\left[1-\frac{i-1}{\tBe}\frac{\chi_\e^{(d-1)}}{\chi_\e^{(d)}}+\mathcal{O}_{\Pp}\left(\tBe^{-2}\right)\right].
\]
\end{proposition}

\begin{proof} The claim is trivial for $i=1$. For $i\ge 2$,
we first observe that \eqref{eq:Duhamel} and \eqref{eq:multipl-with-exp-coordinatewise} imply
\begin{equation}\label{eq:Y-formula-with-dotdot}
Y_{\eps}^{(i)}(t)=\e e^{\lambda t}\left[\frac{t^{d-i}}{(d-i)!}\chi_{\e}^{(d)}(t)+\frac{t^{d-i-1}}{(d-i-1)!}\chi_{\e}^{(d-1)}(t)+\dots\right].
\end{equation}
Plugging in $t=\tBe$ and using \eqref{eq:exit-time-eq} to write
\[
e^{\lambda \tBe}=\frac{\e^{\alpha}(d-1)!}{\e|\chi_\e^{(d)}|\, \tBe^{d-1}\, \left|1+\frac{d-1}{\tBe}\frac{\chi_\e^{(d-1)}}{\chi_{\e}^{(d)}}+\mathcal{O}_{\Pp}\left(\tBe^{-2}\right)\right|},
\]
we obtain 
\[
Y_{\e}^{(i)}(\tBe)=\frac{\e^{\alpha}\sgn(\chi_\e^{(d)})}{\tBe^{i-1}}\frac{(d-1)!}{(d-i)!}\cdot\frac{1+\frac{d-i}{\tBe}\frac{\chi_\e^{(d-1)}}{\chi_\e^{(d)}}+\mathcal{O}_{\Pp}\left(\tBe^{-2}\right)}{1+\frac{d-1}{\tBe}\frac{\chi_\e^{(d-1)}}{\chi_\e^{(d)}}+\mathcal{O}_{\Pp}\left(\tBe^{-2}\right)},\quad i=2,\dots,d-1.
\]
 Using \eqref{eq:elementary-formula} and collecting similar terms yields the desired formula.
For $i=d$,  there is only one term in 
\eqref{eq:Y-formula-with-dotdot}, so
\[
Y_\e^{(d)}(\tBe)=\e e^{\lambda\tBe}\chi_{\e}^{(d)}=\frac{\e^{\alpha}\sgn(\chi_\e^{(d)})(d-1)!}{\tBe ^{d-1}}\frac{1}{1+\frac{d-1}{ \tBe}\frac{\chi_\e^{(d-1)}}{\chi_\e^{(d)}}+\mathcal{O}_{\Pp}\left(\tBe^{-2}\right)},
\]
and the proof finishes by \eqref{eq:elementary-formula} in this case as well.
\end{proof}

\subsection{The exit from $\Vmf$}\label{sec:between-B-eps-and-B}

We now study the process $Y_{\e}$ after the exit from $\Vmf_{\e}$. Namely, we consider $\bar{Y}_\e(t)=Y_\e(t+\tBe)$, which solves the SDE \eqref{eq:linear-SDE} with initial condition given by Proposition \ref{prop:Y-eps-asymp} 
\begin{equation}\label{eq:init-cond-bar}
\bar{Y}^{(i)}_{\e}(0)=\frac{\e^{\alpha}\sgn(\chi_\e^{(d)})}{\tBe^{i-1}}\frac{(d-1)!}{(d-i)!}\left[1-\frac{i-1}{\tBe}\frac{\chi_\e^{(d-1)}}{\chi_\e^{(d)}}+\mathcal{O}_{\Pp}\left(\tBe^{-2}\right)\right], \qquad i=2,\dots,d,
\end{equation}
while $\bar{Y}_\e^{(1)}(0)=\e^{\alpha}\sgn(\chi_\e^{(d)})$.

Our goal is to describe the limiting distribution of the exit time $\bar{\tau}^{\e}$ and the exit location $\bar{Y}_\e(\bar{\tau}^{\e})$ from $\Vmf$.
To this end, let us first prove a result essentially saying that the deterministic dynamics completely dominates the process in this regime.

\begin{lemma}\label{lem:Duhamel-big-box}
For $t\leq\bar{\tau}^{\e}$,  we have
\begin{equation}
\label{eq:Y-bar-main-term-and-error}
\bar{Y}_{\e}(t)=e^{At}\bar{Y}_\e(0)+g_\e(t) ,
\end{equation}
where $g_\e$ is a continuous process such that $\sup_{t\leq\bar{\tau}^\e}|g_\e(t)|=\Oc_{\Pp}(\eps)$.
\end{lemma}

\begin{proof}
Duhamel's formula implies
\begin{equation}
\label{eq:Duhamel-for-bar-process}
\bar{Y}_{\e}(t)=e^{At}\left[\bar{Y}_{\eps}(0)+\e\bar{N}_{\eps}(t)+\e^2 \bar{D}_{\eps}(t)\right],
\quad t\le \bar\tau_\eps,
\end{equation}
where
\[
\bar{N}_{\eps}(t)=\int_0^t e^{-As}\tilde{\sigma}(\bar{Y}_{\eps}(s))dW(s),\qquad \bar{D}_{\eps}(t)=\frac{1}{2}\int_0^t e^{-As}a(\bar{Y}_{\eps}(s))ds.
\]
Repeating the proof of Lemma \ref{lem:a-priori-N-D}, one can prove
\begin{equation}
\sup_{\e\geq 0}\mathrm{P}\left\{\sup_{t\leq\bar{\tau}^\e}\|\bar{N}_{\e}(t)\|_{\infty}>z\right\}\leq\frac{c}{z^2}
,\qquad\sup_{\eps>0,\,t\leq\bar{\tau}^\e}\|\bar{D}_{\eps}(t)\|_{\infty}\leq D_0,
\end{equation}
implying that $\sup_{t\leq\tau_{\Vmf}^{\e}}\left[|\e\bar{N}_{\eps}(t)+\e^2\bar{D}_{\eps}(t)|\right]=\Oc_{\Pp}(\eps)$. 
\end{proof}

By~\eqref{eq:init-cond-bar} and Proposition~\ref{prop:tau-B-eps-asymp}, we see that the error term $g_{\e}(t)$ in \eqref{eq:Y-bar-main-term-and-error} does not play any role in the scaling limit of $\bar Y_\eps$.

\begin{lemma}\label{lem:Y-bar-i-asymp-t}
For $t\leq\bar \tau^\eps$ and $i=1,\dots,d$,
\begin{multline*}
\bar{Y}_\e^{(i)}(t)=\frac{\e^{\alpha}e^{\lambda t}\sgn(\chi_\e^{(d)})}{\tBe^{i-1}}\left(1+\frac{t}{\tBe}\right)^{d-i}\frac{(d-1)!}{(d-i)!}\times\\ \times
\left[1-\frac{1}{\tBe}\frac{\chi_\e^{(d-1)}}{\chi_{\e}^{(d)}}\frac{(i-1)\tBe+(d-1)t}{\tBe+t}+\mathcal{O}_{\Pp}\left(\tBe^{-2}\right)\right].
\end{multline*}
\end{lemma}

\begin{proof}
By \eqref{eq:multipl-with-exp-coordinatewise},  the $i$th coordinate of $e^{At}\bar{Y}(0)$ can be written as
\begin{align}\label{eq:Y-bar-i-expand}
e^{\lambda t}&\left[\sum_{j=0}^{d-i}\frac{t^j}{j!}\bar{Y}_\e^{(i+j)}(0)\right]\\
&\notag=\frac{\e^{\alpha}e^{\lambda t}\sgn(\chi_\e^{(d)})}{\tBe^{i-1}}\frac{(d-1)!}{(d-i)!}\left[\sum_{j=0}^{d-i}\left(\frac{t}{\tBe}\right)^j\frac{(d-i)!}{j!(d-i-j)!}\left(1-\frac{i+j-1}{\tBe}\frac{\chi_\e^{(d-1)}}{\chi_{\e}^{(d)}}\right)+\mathcal{O}_{\Pp}\left(\tBe^{-2}\right)\right].
\end{align}
After a little manipulation, the sum in the bracket can be written as
\begin{multline*}
\left(1-\frac{(i-1)}{\tBe}\frac{\chi_\e^{(d-1)}}{\chi_{\e}^{(d)}}\right)\sum_{j=0}^{d-i}\binom{d-i}{j}\left(\frac{t}{\tBe}\right)^j-\frac{(d-i)}{\tBe}\frac{\chi_\e^{(d-1)}}{\chi_{\e}^{(d)}}\frac{t}{\tBe}\sum_{j=0}^{d-i-1}\binom{d-i-1}{j}\left(\frac{t}{\tBe}\right)^j
\\
=
\left(1-\frac{(i-1)}{\tBe}\frac{\chi_\e^{(d-1)}}{\chi_{\e}^{(d)}}\right)\left(1+\frac{t}{\tBe}\right)^{d-i}-\frac{(d-i)}{\tBe}\frac{\chi_\e^{(d-1)}}{\chi_{\e}^{(d)}}\frac{t}{\tBe}\left(1+\frac{t}{\tBe}\right)^{d-i-1}
\\
=
\left(1+\frac{t}{\tBe}\right)^{d-i}\left[1-\frac{1}{\tBe}\frac{\chi_\e^{(d-1)}}{\chi_{\e}^{(d)}}\frac{(i-1)\tBe+(d-1)t}{\tBe+t}\right].
\end{multline*}
The result now follows by plugging this back into \eqref{eq:Y-bar-i-expand}  and using Lemma \ref{lem:Duhamel-big-box}.
\end{proof}

The next lemma shows that with probability close to one, the exit from $\Vmf$ happens close to $\pm R e_1$ and that up until this exit $Y_\e$ follows closely the corresponding deterministic orbit contained in the subspace generated by $e_1$. We introduce
\[
\bar{\tau}_{i}^{\eps}=\inf\{t>0: |\bar{Y}_{\eps}^{(i)}(t)|=R\}
\]
and let $\bar{\tau}^{\e}=\min\{\bar{\tau}_i^{\e}\}$.

\begin{proposition}\label{prop:tau-bar-one-asymp}
As $\e\downarrow 0$, we have 
\begin{align}
&\mathrm{P}\{\bar{\tau}^\e\neq \bar{\tau}_1^{\e}\}\to 0,
\label{eq:exit-in-the-right-direction}
\\
\bar{\tau}^{\e}=\frac{\alpha}\lambda\log\e^{-1}&+\frac{d-1}{\lambda}\log(1-\alpha)+\frac{1}{\lambda}\log R+\frac{\bar{K}(\e)}{\lambda},
\label{eq:asympt_of_bar_tau}
\end{align}
where 
\begin{multline}\label{eq:final-expr-for-Kbar}
\bar{K}(\e)=-\frac{\alpha(d-1)^2}{1-\alpha}\frac{\log\log\e^{-1}}{\log\e^{-1}}-\frac{(d-1)\log R}{\log\e^{-1}}-\frac{\alpha}{1-\alpha}\frac{(d-1)\log\frac{|\chi_\e^{(d)}|}{(d-1)!\lambda^{d-1}}}{\log\e^{-1}}
\\ -
\frac{(d-1)^2}{1-\alpha}\frac{\log(1-\alpha)}{\log\e^{-1}}+\frac{\alpha(d-1)\lambda}{(1-\alpha)\log\e^{-1}}\frac{\chi_\e^{(d-1)}}{\chi_\e^{(d)}}
+o_{\Pp}\left(\frac{1}{\log\e^{-1}}\right).
\end{multline}
In particular, 
\begin{equation}
\label{eq:barK-o1}
\bar{K}(\e)=o_{\Pp}(1).
\end{equation} 
Moreover,
\begin{equation}
\label{eq:following-deterministic-traj}
\sup_{t\leq\bar{\tau}^\e}\left|\bar{Y}_{\e}(t)-\e^{\alpha}\sgn(\chi_\e^{(d)})e^{\lambda t}\left(1+\frac{\lambda t}{(1-\alpha)\log\e^{-1}}\right)^{d-1}e_1\right|\stackrel{\Pp}{\to} 0,\quad 
\e\downarrow 0.
\end{equation}
\end{proposition}

\begin{proof}
It immediately follows from Lemma \ref{lem:Y-bar-i-asymp-t} that
\[
\sup_{t\leq\bar{\tau}^{\e}}\frac{\bar{Y}_\e^{(i)}(t)}{\bar{Y}_\e^{(1)}(t)}\stackrel{\Pp}{\to} 0,\qquad \e\downarrow 0.
\]
which proves~\eqref{eq:exit-in-the-right-direction}.

By the definition of $\bar{\tau}^{\e}$, we have $|Y^{(1)}_\e(\bar{\tau}^{\e})|=R$ with probability close to one and thus we can use Lemma \ref{lem:Y-bar-i-asymp-t} for $i=1$  to obtain
\begin{equation}\label{eq:tau-bar-eq}
R=\e^{\alpha}e^{\lambda\bar{\tau}^\e}\left(\frac{\tBe+\bar{\tau}^{\e}}{\tBe}\right)^{d-1}
\left|1-\frac{d-1}{\tBe}\frac{\chi_\e^{(d-1)}}{\chi_{\e}^{(d)}}\frac{\bar{\tau}^{\e}}{\tBe+\bar{\tau}^{\e}}+\mathcal{O}_{\Pp}\left(\tBe^{d-2}\right)\right|.
\end{equation}
We define $\bar K(\eps)$ by \eqref{eq:asympt_of_bar_tau} which is equivalent to 
\begin{equation}
\label{eq:seek-solution-for-bar-tau}
\lambda\bar{\tau}^{\e}=\log\frac{R}{\e^{\alpha}}+(d-1)\log(1-\alpha)+\bar{K}(\e)\quad\textrm{\ or}\quad e^{\lambda\bar{\tau}^{\e}}=\e^{-\alpha}R(1-\alpha)^{d-1}e^{\bar{K}(\e)}.
\end{equation}
Using this along with \eqref{eq:asymp-exit-time},\eqref{eq:K-eps-o1},
 we obtain
\begin{equation}
\lambda(\tBe+\bar{\tau}^\e)=\log\e^{-1}-(d-1)\log\log\e^{-1}-\log\frac{|\chi_\e^{(d)}|}{R(d-1)!\lambda^{d-1}}+o_{\Pp}(1)+\bar K(\eps).
\label{eq:sum-of-taus}
\end{equation}
Plugging~\eqref{eq:seek-solution-for-bar-tau} and~\eqref{eq:sum-of-taus} into \eqref{eq:tau-bar-eq}, using $\chi_{\e}\stackrel{\mathrm{d}}{\to}\chi$ and \eqref{eq:asymp-exit-time},\eqref{eq:K-eps-o1}, we obtain
\begin{equation*}
1=e^{\bar{K}(\e)}\left(1+o_\Pp(1)+\frac{\bar K(\eps)}{\ln\eps^{-1}}(1+o_{\Pp}(1))\right).
\end{equation*}
Therefore,~\eqref{eq:barK-o1} holds. Using it in~\eqref{eq:sum-of-taus} gives
\begin{equation}
\label{eq:total-exit-time-representation}
\lambda(\tBe+\bar{\tau}^\e)=\log\e^{-1}-(d-1)\log\log\e^{-1}-\log\frac{|\chi_\e^{(d)}|}{R(d-1)!\lambda^{d-1}}+o_{\Pp}(1).
\end{equation}
Now 
a straightforward calculation based on \eqref{eq:asymp-exit-time}, \eqref{eq:tau-bar-eq}, \eqref{eq:seek-solution-for-bar-tau}, \eqref{eq:total-exit-time-representation}, and \eqref{eq:elementary-formula} reveals
\begin{multline*}
\ 1=e^{\bar{K}(\e)}\Bigg[1+\frac{\alpha}{1-\alpha}\frac{(d-1)^2\log\log\e^{-1}}{\log\e^{-1}}+(d-1)\frac{\log R}{\log\e^{-1}}
+\frac{\alpha}{1-\alpha}\frac{(d-1)\log\left[\frac{|\chi_\e^{(d)}|}{(d-1)!\lambda^{d-1}}\right]}{\log\e^{-1}}\\ +\frac{(d-1)^2}{1-\alpha}\frac{\log(1-\alpha)}{\log\e^{-1}}-\frac{\alpha}{1-\alpha}\frac{(d-1)\lambda}{\log\e^{-1}}\frac{\chi_\e^{(d-1)}}{\chi_\e^{(d)}}+o_{\Pp}\left(\frac{1}{\log\e^{-1}}\right)\Bigg].
\end{multline*}
Another application of \eqref{eq:elementary-formula} finishes the proof of~\eqref{eq:final-expr-for-Kbar}. 

To prove~\eqref{eq:following-deterministic-traj}, we note that Lemma \ref{lem:Y-bar-i-asymp-t} and the result on $\bar{\tau}^{\e}$ implies
\[
\bar{Y}_{\e}^{(1)}(t)=\e^{\alpha}e^{\lambda t}\sgn(\chi_\e^{(d)}) \left(1+\frac{t}{\tBe}\right)^{d-1}\left(1+o_{\Pp}(1)\right),\qquad\sup_{t\leq\bar{\tau}^{\e}}\bar{Y}_\e^{(i)}(t)=o_{\Pp}(1),\qquad i=2,\dots,d.
\] 
Combining this with \eqref{eq:asymp-for-tilde-tau}, we complete the proof.
\end{proof}
\\

We are ready to finish the proof of the result on the linear system exiting from the box $\Vmf$.

\smallskip

\noindent\textit{Proof of Theorem \ref{thm:linear-result}}.
Clearly, $\tau_{\Vmf}^{\e}=\tBe+\bar{\tau}^{\e}$ and as a simple consequence of Proposition \ref{prop:tau-B-eps-asymp} and Proposition \ref{prop:tau-bar-one-asymp}, we have
\begin{equation}\label{eq:tau-B-e-full-asymp}
\tau_{\Vmf}^\e=\frac{1}{\lambda}\log\e^{-1}-\frac{d-1}{\lambda}\log\log\e^{-1}-\frac{1}{\lambda}\log\frac{|\chi_\e^{(d)}|}{R(d-1)!\lambda^{d-1}}+\frac{K(\e)}{\lambda},
\end{equation}
where
\[
K(\e)=\frac{(d-1)^2\log\log\e^{-1}}{\log\e^{-1}}+\frac{d-1}{\log\e^{-1}}\eta_{\e}+o_{\Pp}\left(\frac{1}{\log\e^{-1}}\right),\quad{\ } \eta_{\e}=-\lambda\frac{\chi_\e^{(d-1)}}{\chi_{\e}^{(d)}}+\log\left[\frac{|\chi_\e^{(d)}|}{R(d-1)!\lambda^{d-1}}\right].
\]

Combining \eqref{eq:tau-bar-eq} with Lemma \ref{lem:Y-bar-i-asymp-t}, we obtain
\begin{align*}
Y_\e^{(i)}(\tau_\Vmf^\e)&=\frac{R\sgn(\chi_\e^{(d)})}{\tBe^{i-1}}\left(1+\frac{\bar{\tau}^\e}{\tBe}\right)^{-(i-1)}\frac{(d-1)!}{(d-i)!}\frac{1-\frac{1}{\tBe}\frac{\chi_\e^{(d-1)}}{\chi_{\e}^{(d)}}\frac{(i-1)\tBe+(d-1)\bar{\tau}^\e}{\tBe+\bar{\tau}^\e}+\Oc_{\Pp}\left(\tBe^{-2}\right)}{1-\frac{1}{\tBe}\frac{\chi_\e^{(d-1)}}{\chi_{\e}^{(d)}}\frac{(d-1)\bar{\tau}^{\e}}{\tBe+\bar{\tau}^{\e}}+\Oc_{\Pp}\left(\tBe^{-2}\right)}
\\
&=
\frac{R\sgn(\chi_\e^{(d)})}{\left[\tau_\Vmf^\e\right]^{i-1}}\frac{(d-1)!}{(d-i)!}\left[1-\frac{i-1}{\tau_\Vmf^\e}\frac{\chi_\e^{(d-1)}}{\chi_{\e}^{(d)}}+\Oc_{\Pp}\left(\tBe^{-2}\right)\right],
\end{align*}
where we used \eqref{eq:elementary-formula} in the second equality. A straightforward calculation using the asymptotics of $\tau_{\Vmf}^\e$ and \eqref{eq:elementary-formula} implies

\begin{equation}\label{eq:precise-formula-in-all-directions}
Y_\e^{(i)}(\tau_\Vmf^\e)=\frac{\lambda^{i-1}R\sgn(\chi_\e^{(d)})}{\log^{i-1}\e^{-1}}\frac{(d-1)!}{(d-i)!}\left[1+\frac{i-1}{\log\e^{-1}}\left((d-1)\log\log\e^{-1}+\eta_\e\right)+o_{\Pp}\left(\frac{1}{\log\e^{-1}}\right)\right].
\end{equation}
Using this identity for $i=1,2,3$, we obtain
\begin{multline}
\label{eq:asymptotics-of-exit-from-small-in-proof}
Y_\e(\tau_\Vmf^\e)=R\sgn(\chi_\e^{(d)})\Biggl[e_1+
\lambda(d-1)\left(\frac{1}{\log \eps^{-1}}
+\frac{(d-1) \log\log\e^{-1}}{\log^2\e^{-1}}
-\frac{\eta_\eps}{\log^2\e^{-1}}\right)e_2
\\+\frac{\lambda^2(d-1)(d-2)}{\log^2\eps^{-1}}e_3+o_\Pp\left(\frac{1}{\log^2\eps^{-1}}\right)\Bigg].
\end{multline}

Due to~\eqref{eq:chi-converges} and $\Pp\{\chi^{(d)}=0\}=0$, we can conclude that
$\Pp\{\sgn(\chi_\eps^{(d)})=\sgn(\chi^{(d)})\}\to 1$ as $\eps\to 0$. Using this, we see that on 
$A^{\pm}=\{\pm\chi^{(d)}>0\}$,
the expansions~\eqref{eq:tau-B-e-asymp-in-theorem}
and~\eqref{eq:asymptotics-of-exit-from-small} of the theorem 
follow 
from \eqref{eq:tau-B-e-full-asymp},\eqref{eq:asymptotics-of-exit-from-small-in-proof}, and 
\begin{equation*}
\left(\chi_{\e}^{(d-1)},\chi_\e^{(d)},\sgn(\chi_\e^{(d)}), \eta_{\e}\right)\stackrel{\Pp}{\to}\left(\chi^{(d-1)},\chi^{(d)},\sgn(\chi^{(d)}), \eta\right),\quad \e\downarrow 0.
\end{equation*}
Also,~\eqref{eq:motion-is-close-to-axis} follows from Proposition~\ref{prop:tau-bar-one-asymp}. 
\qed

\bibliographystyle{Martin}
\bibliography{citations} 
\end{document}